\documentclass[12pt]{article}
\usepackage{amsthm,amssymb,amsmath,graphicx,cite}
\usepackage{algpseudocode,algorithm,algorithmicx}
\usepackage[margin = 1in]{geometry}
\usepackage{multirow}
\usepackage{float}
\usepackage{fullwidth}
\usepackage{placeins}
\usepackage[margin=1cm]{caption}
\usepackage{makecell}
\usepackage{multirow}
\usepackage{float}
\setcellgapes{4pt}
\usepackage{subfig}
\usepackage{color}

\newcommand{\dmax}{\displaystyle\max}
\newcommand{\1}{\mathbf{1}}
\renewcommand{\H}{\mathcal{H}}

\newcommand{\ip}[2]{\left\langle #1 , #2 \right\rangle}    
\newcommand{\R}{{\mathbb R}}

\DeclareMathOperator*{\argmin}{arg\,min}

\DeclareMathOperator{\dist}{dist}

\newtheorem{theorem}{Theorem}
\newtheorem{proposition}{Proposition}

\def\transp{^{\text{\sf T}}}
\newcommand{\matr}[1]{\begin{bmatrix} #1 \end{bmatrix}}    

\title{An easily computable upper bound on the Hoffman constant for homogeneous inequality systems}
\author{
Javier F. Pe\~na\thanks{Tepper School of Business,
Carnegie Mellon University, USA, {\tt jfp@andrew.cmu.edu}}
}

\begin{document}

\maketitle

\begin{abstract}
Let $A\in \R^{m\times n}\setminus \{0\}$ and $P:=\{x:Ax\le 0\}$. This paper provides  a procedure to compute an upper bound on the following {\em homogeneous Hoffman constant}
\[
H_0(A) := \sup_{u\in \R^n \setminus P} \frac{\dist(u,P)}{\dist(Au, \R^m_-)}.
\]  
In sharp contrast to the intractability of computing more general Hoffman constants, the procedure described in this paper is entirely tractable and easily implementable. 
\end{abstract}
\section{Introduction}

Hoffman constants for systems of linear inequalities, and more general error bounds for feasibility problems, play a central role in mathematical programming.  In particular, Hoffman constants provide a key building block for the convergence of a variety of algorithms~\cite{ApplHLL22,BeckS17,LacoJ15,LeveL10,LuoT93,WangL14}.  Since Hoffman's seminal work~\cite{Hoff52}, Hoffman constants and more general error bounds has been widely studied~\cite{AzeC02,burke1996,guler1995,Li93,MangS87,robinson1973,Zali03,zheng2004}.  However, there has been very limited work on algorithmic procedures that compute or bound Hoffman constants.  The only two references that appear to tackle this computational challenge are the 1995 article by Klatte and Thiere~\cite{KlatT95} and the more recent 2021 article by Pe\~na, Vera, and Zuluaga~\cite{PenaVZ21}.   However, as it is discussed in both~\cite{KlatT95} and~\cite{PenaVZ21}, there are limitations on the algorithmic schemes proposed in both these articles.  

\medskip

The central goal of this paper is to devise a procedure that computes an upper bound on the following {\em homogeneous Hoffman constant} $H_0(A)$.  Suppose $A\in \R^{m\times n}$.  Let $P:=\{x:Ax\le 0\}$ and define $H_0(A)$ as 
\[
H_0(A):=\sup_{u\in \R^n \setminus P} \frac{\dist(u,P)}{\dist(Au, \R^m_-)}.
\]
For notational convenience, by convention let $H_0(A) := 0$  when $P=\R^n$.  This occurs precisely when $A = 0$.

To position this work in the context of Hoffman constants, we next recall the {\em local} and {\em global} Hoffman constants $H(A,b)$ and $H(A)$ associated to linear systems of inequalities defined by $A$.  The homogeneous Hoffman constant $H_0(A)$ is a special case of the following {\em local Hoffman constant} $H(A,b)$.  Suppose $A\in \R^{m\times n}$ and 
$b\in A\R^n + \R^m_+$. Let $P_A(b):=\{x\in \R^n: Ax\le b\}$ and define $H(A,b)$ as
\[
H(A,b):=\sup_{u \in \R^n\setminus P_A(b)} \frac{\dist(u,P_A(b))}{\dist(Au, \R^m_-)}.
\]
It is evident that $H_0(A) = H(A,0)$ and thus $H_0(A)$ is bounded above by the following {\em global Hoffman constant} $H(A)$.  Suppose $A\in \R^{m\times n}$.  Define
\[
H(A):=\sup_{b\in A\R^n + \R^m_+} H(A,b).
\] 
In his seminal paper~\cite{Hoff52}, Hoffman showed that $H(A)$ is finite and consequently so are $H_0(A)$ and $H(A,b)$ for all $b\in A\R^n + \R^m_+$.

The articles~\cite{KlatT95,PenaVZ21} propose algorithms to compute or estimate the global Hoffman constant $H(A)$.  These algorithms readily yield a computational procedure to bound $H_0(A)$.  However, as it is detailed in~\cite{KlatT95,PenaVZ21}, except for very special cases the computation or even approximation of $H(A)$ is an extremely challenging problem.  Indeed, the recent results in~\cite{PenaVZ2019} show that the Stewart-Todd condition measure $\chi(A)$~\cite{Stew89,Todd90} is the same as $H(\mathbf{A})$ where $\mathbf{A} = \matr{A\\-A}$.  Since the quantity $\chi(A)$ is known to be NP-hard to approximate~\cite{Khac95}, so is $H(A)$. The computation of the (non-homogeneous) local Hoffman constant $H(A,b)$, as discussed in~\cite{AzeC02,zheng2004}, also poses similar computational challenges. In sharp contrast, the procedure proposed in this paper for upper bounding the more specialized Hoffman constant $H_0(A)$ is entirely tractable and easily implementable for any $A\in \R^{m\times n}$.  
The bound is a formalization of the following three-step approach detailed in Section~\ref{sec.upper.bounds}.

\medskip

First, upper bound $H_0(A)$ in the following two special cases:
\begin{itemize}\item[(i)] 
When $A\hat x < 0$ for some $\hat x \in \R^n$ or equivalently when $A\transp y = 0, y\ge 0 \Rightarrow y =0$.  (See Proposition~\ref{prop.case.1}.)
\item[(ii)] When $A\transp \hat y = 0$ for some $\hat y > 0$ or equivalently when
$Ax\le 0\Rightarrow Ax=0$. (See Proposition~\ref{prop.case.2}.)
\end{itemize}

\medskip

Second, use a canonical partition $A = \matr{A_B\\ A_N}$ of the rows of $A$ such that $A_N$ is as in case (i) and $A_B$ is as in case (ii) above.  (See Proposition~\ref{prop.partition}.) 

\medskip

Third, upper bound $H_0(A)$ by stitching together the Hoffman constants $H_0(A_B)$, $H_0(A_N),$ and a third Hoffman constant  $\H(L,K)$ associated to the intersection of the subspace $L:=\{x: A_Bx = 0\}$ and the cone $K:=\{x: A_N x \le 0\}$. (See Theorem~\ref{main.thm}.)

\medskip

The above steps suggest the following computational procedure to upper bound $H_0(A)$:
First, compute the partition $B,N$.  Second, compute upper bounds on $H_0(A_B)$ and on $H_0(A_N)$.  Third, upper bound $\H(L,K)$.
Section~\ref{sec.computation} details this procedure.  As explained  in Section~\ref{sec.computation}, the total computational work in the entire procedure consists of two linear programs, two quadratic programs, a convex program, and a singular value calculation, all of which are computationally tractable.  
This is noteworthy in light of the challenges associated to estimating the Hoffman constants $H(A)$ and $H(A,b)$. 
A Python implementation and some illustrative examples of this procedure are publicly available at
\begin{center}
{\tt https://github.com/javi-pena} 
\end{center}

For ease of notation and computability, we assume throughout the paper that the norm in $\R^m$ satisfies the following {\em componentwise compatibility condition:} if $y,z\in \R^m$ and  $|y|\le |z|$ componentwise then $\|y\| \le \|z\|$.  The componentwise compatibility condition in particular implies that for all $u\in\R^n$
\[
\dist(Au,\R^n_-) = \|(Au)^+\|
\]
where $(Au)^+ = \max\{Au,0\}$ componentwise.  Consequently, 
\[
H_0(A)=\sup_{u\in \R^n \setminus P} \frac{\dist(u,P)}{\|(Au)^+\|}.
\]
Observe that most of the usual norms in $\R^m$, including the $\ell_p$ norms for $1\le p \le \infty$ satisfy the componentwise compatibility  condition.

We conclude this introduction by highlighting that our developments for bounding $H_0(A)$ rely critically on the features of  homogeneous systems of inequalities.  In contrast to non-homogeneous systems of inequalities and more general affine cone inclusions, homogeneous systems of inequalities and more general homogeneous affine cone inclusions possess a number of attractive properties as discussed in~\cite{BurkD09,Robi72,Robi76,Urse75}.  
In particular, although it is tempting to conjecture that a bound on the non-homogeneous Hoffman constant $H(A,b)$ could be obtained from some $H_0(A_b)$ via homogenization, that is not the case as we next detail.  Indeed, consider the natural homogenization $A_b z \le 0$ of the system of inequalities $Ax\le b$ where 
\[
A_b := \matr{A&-b\\0&-1}, \; z := \matr{y\\t}.
\]
The following  example shows that $H(A,b)$ cannot be bounded above by any reasonable multiple of $H_0(A_b)$. Suppose $0 < \epsilon < 1$ and let
\[
A = \matr{1 & \epsilon \\
-1 & \epsilon \\
0& -1 }, \; b = \matr{1\\1\\0}.
\]
Then
\[
A_b = \matr{1 & \epsilon & -1\\
-1 & \epsilon & -1\\
0& -1 & 0 \\
0&0&-1}.
\]
For ease of computation, suppose all relevant spaces are endowed with the infinite norm.
Hence the remarks following Proposition~\ref{prop.case.1} below imply that $H_0(A_b) \le 1$.  On the other hand, $H(A,b) \ge 1/\epsilon$ because $Ax \le b$ implies that $x_2\le 1/\epsilon$ and thus for $\tilde x = \matr{0,2/\epsilon}$ we have $\|(A\tilde x-b)_+\|_{\infty}=1$ but $\|\tilde x-x\|_{\infty}\ge 1/\epsilon = 1/\epsilon \cdot \|(A\tilde x-b)_+\|_{\infty}$ for any $x$ such that $Ax \le b$.  Since this holds for any $0<\epsilon < 1$, it follows that $H(A,b)$ cannot be bounded above in terms of $H_0(A_b)$.

\section{Upper bounds on $H_0(A)$}
\label{sec.upper.bounds}
\subsection{Upper bounds on $H_0(A)$ in two special cases}

We next consider two special cases that can be seen as dual counterparts of each other.

\begin{proposition}\label{prop.case.1}
Suppose $A\in \R^{m\times n}$ and $A\hat x < 0$ for some $\hat x\in\R^n$  or equivalently $A\transp y = 0, y \ge 0 \Rightarrow y=0$. Then
\begin{equation}
\label{eq.case.1}
H_0(A) \le \max_{y\in\R^m\atop\|y\|\le 1} \min_{x\in \R^n\atop Ax\le y} \|x\|.
\end{equation}
\end{proposition}
\begin{proof}
For ease of notation, let $H$ denote the right-hand side expression in~\eqref{eq.case.1}, that is,
\[
H := \max_{y\in\R^m\atop\|y\|\le 1} \min_{x\in \R^n\atop Ax\le y} \|x\| = \max_{y\in\R^m\setminus\{0\}} \min_{x\in \R^n\atop Ax\le y} \frac{\|x\|}{\|y\|} .
\]
Observe that $H < +\infty$ because the assumption on $A$ implies that $A\R^n + \R^m_+ = \R^m$. 

We need to show that $H_0(A) \le H$.  To that end, let $P:=\{x\in \R^n: Ax \le 0\}$ and suppose that $u \in \R^n\setminus P$.  Let $y:=(Au)^+ \in \R^m$.  The construction of $H$ implies that there exists $x \in \R^n$ such that $Ax\le -y$ and $\|x\| \le H \cdot \|y\| = H \cdot \|(Au)^+\|$.  Thus $x+u \in P$ because
$$A(x+u) = Ax + Au \le -y + Au = -(Au)^+ + Au\le 0.$$
Furthermore $\|(x+u) - u\| = \|x\| \le H \cdot \|(Au)^+\|$. Since this holds for all $u\in\R^n\setminus P$, it follows that $H_0(A) \le H$.
\end{proof}
In addition to the simple direct proof above, an alternative proof of Proposition~\ref{prop.case.1} can also be obtained from~\cite{PenaVZ21}.  Indeed,~\cite[Proposition 2]{PenaVZ21} implies that when $A\in \R^{m\times n}$ satisfies the assumption in Proposition~\ref{prop.case.1}, the right-hand side in~\eqref{eq.case.1} is precisely the global Hoffman constant $H(A)$ which is at least as large as $H_0(A)$ as previously noted.

For computational purposes, it is useful to note that when $\R^m$ is endowed with the $\ell_\infty$ norm, the upper bound in Proposition~\ref{prop.case.1} can be computed via the following convex optimization problem:
\[
\min\{ \|x\|:  Ax \ge \1\}.
\]
In particular, any $\bar x \in \R^n$ such that $A\bar x \ge \1$ yields the upper bound $$H_0(A) \le \|\bar x\|.$$ 

\bigskip

The following proposition, which can be seen as a dual counterpart of Proposition~\ref{prop.case.1}, relies on the dual norms in $\R^m$ and $\R^n$.  More precisely, suppose both $\R^m$ and $\R^n$ are endowed with their canonical inner products.  In each case let $\|\cdot\|^*$ denote the norm defined as
\[
\|u\|^* = \max_{\|x\|\le 1} \ip{u}{x}.
\]

\begin{proposition}\label{prop.case.2}
Suppose $A\in \R^{m\times n}$ is such that $A\transp\hat y = 0$ for some $\hat y >0 $  or equivalently $Ax \le 0 \Rightarrow Ax = 0$. Then
\begin{equation}
\label{eq.case.2}
H_0(A) \le 
\max_{v \in A\transp(\R^m) \atop \|v\|^*\le 1} \min_{y\in \R^m_+, A\transp y = v} \|y\|^*.
\end{equation}
\end{proposition}
\begin{proof} We shall assume that $A\ne 0$ as otherwise $H_0(A) = 0$ and~\eqref{eq.case.2} trivially holds.
Again for ease of notation, let $H$ 
denote the right-hand side expression in~\eqref{eq.case.2}, that is,
\[
H := \max_{v \in A\transp(\R^m) \atop \|v\|^*\le 1} \min_{y\in \R^m_+, A\transp y = v} \|y\|^* = \max_{v \in A\transp(\R^m) \atop v \ne 0} \min_{y\in \R^m_+, A\transp y = v}  \frac{\|y\|^*}{\|v\|^*}.
\]
Observe that $H < +\infty$ because the assumption on $A$ implies that $A\transp \R^m_+= A\transp\R^m$. 

We need to show that $H_0(A) \le H$.  To that end, let $P:=\{x\in \R^n: Ax \le 0\}= \{x \in \R^n: Ax = 0\}$ and suppose that $u \in \R^n\setminus P$.  Let 
\[
\bar x := \argmin_{x \in P} \|u-x\| = \argmin_{x: Ax=0} \|u-x\|.
\]
The optimality conditions of the latter problem imply that there exists $v\in A\transp \R^m$ with $\|v\|^*=1$ such that
\[
\|u-\bar x\| = \ip{v}{u-\bar x}.
\]
The construction of $H$ implies that there exists $y \in \R^m_+$ such that $A\transp y =v$ and $\|y\|^* \le H.$  Since $v = A\transp y$ we have 
\[
\|u-\bar x\| = \ip{v}{u-\bar x} = \ip{A\transp y}{u-\bar x} = \ip{y}{A(u-\bar x)}
=\ip{y}{Au}.
\]
In addition, since $y \in \R^m_+$ and $\|y\|^* \le H$, we also have
\[
\|u-\bar x\| = \ip{y}{Au} \le \ip{y}{(Au)^+} \le \|y\|^* \cdot\|(Au)^+\| \le H \cdot \|(Au)^+\|.
\]
Since this holds for all $u\in\R^n\setminus P$, it follows that $H_0(A) \le H$.
\end{proof}

For computational purposes, it is useful to note that when $\R^m$ is endowed with the $\ell_\infty$ norm, the upper bound in Proposition~\ref{prop.case.2} can be computed as follows
\[
\max_{v\in A\transp(\R^m) \atop \|v\|^*\le 1} \min_{y\in \R^m_+\atop A\transp y = v} \1\transp y.
\]
The reciprocal of the latter quantity in turn is the radius of the largest ball in $A\transp(\R^m)$ centered at $0$ and contained in the set 
\[
\{A\transp y: y \in \R^m_+,  \1\transp y = 1\} = \{A\transp y: y \in \R^m_+,  \1\transp y \le 1\}.
\]
Therefore, if in addition $\R^n$ is endowed with the $\ell_2$ norm then any $\bar y\in \R^m_{++}$ with $\1\transp \bar y =1$ and $A\transp \bar y = 0$ yields the upper bound
\begin{equation}\label{eq.upper.2}
H_0(A) \le \frac{2}{\sigma_{\min}^+(A\transp \bar  Y)},
\end{equation}
where $\bar Y = \text{Diag}(\bar y)$ and $\sigma_{\min}^+(A\transp \bar Y)$ denotes the smallest positive singular value of $A\transp \bar Y$.
To see why~\eqref{eq.upper.2} holds, observe that if $v \in A\transp \R^m$ and $\|v\|_2 \le \frac{\sigma_{\min}^+(A\transp \bar Y)}{2}$ then $2v = A\transp \bar Y z$ for some $\|z\|_2\le 1$. The latter implies that $|\bar Yz| \le \bar y$ componentwise and thus $2v = A\transp(\bar y + \bar Yz)$ with
\[
\bar y + \bar Yz \in \R^m_+ \text{ and } \1\transp(\bar y+\bar Yz) \le 2\cdot \1\transp \bar  y =2.
\]
In particular, $v\in \{A\transp y: y\in \R^m_+, \1\transp y \le 1\}$.  Since this holds for any $v\in A\transp \R^m$ with $\|v\|_2 \le \frac{\sigma_{\min}^+(A\transp \bar Y)}{2}$, it follows that the radius of the largest ball in  $A\transp(\R^m)$ centered at $0$ and contained in the set 
\[
\{A\transp y: y \in \R^m_+,  \1\transp y = 1\} = \{A\transp y: y \in \R^m_+,  \1\transp y \le 1\}.
\]
is at least $\frac{\sigma_{\min}^+(A\transp \bar Y)}{2}$.

\subsection{Upper bound on $H_0(A)$ for general $A$}

An upper bound on $H(A)$ for general $A\in \R^{m\times n}$ follows by stitching together the cases in the above two propositions via the the canonical partition result in Proposition~\ref{prop.partition} and 
the additional Hoffman constant $\H(L,K)$ defined in~\eqref{eq.hoffman} below.  

The following result is a consequence of the classical Goldman-Tucker partition theorem.  To make our exposition self-contained, we include a proof. 
  
\begin{proposition}\label{prop.partition} Let $A\in \R^{m\times n}$. There exists a unique partition $B\cup N = \{1,\dots,m\}$ such that 
\[
A_B \hat x = 0, \; A_N\hat x < 0 \text{ for some } \hat x\in \R^n
\]
and
\[
A_B\transp  \hat y_B = 0 \text{ for some } \hat y_B > 0.
\]
\end{proposition}
\begin{proof}
Let $N\subseteq\{1,\dots,m\}$ be the largest subset of $\{1,\dots,m\}$ such that
\[
Ax \le 0 \text{ and } A_Nx < 0
\]
has a solution.  
In other words,
\[
N:=\{i\in \{1,\dots,m\}: Ax\le 0 \text{ and } (Ax)_i < 0 \text{ for some } x\in \R^n\}.
\]
Observe that $N$ is well-defined and unique and thus so is $B:=\{1,\dots,m\}\setminus N$.  Furthermore the construction of $N$ implies that $Ax\le 0$ and $A_Nx<0$ for some $x\in\R^n$. Hence   
to finish the proof it suffices to show that 
\[A_B\transp y_B  = 0,\; y_B > 0\] 
has a solution.
To that end, for $i\in\{1,\dots,m\}$ let $e_i\in \R^n$ is the vector with $i$-th component equal to one and all other equal to zero.  Observe that $i \in B$ if and only if the following system of equations and inequalities does not have a solution:
\[
\matr{ A & e_i}\matr{x\\t} \le 0, \matr{0 & 1}\matr{x\\t} > 0.
\]
Farkas Lemma thus implies that $i\in B$ if and only if the following system of equations and inequalities has a solution:
\[
\matr{A\transp \\ e_i\transp} y = \matr{0 \\ 1},  y \ge 0.
\]
Since this holds for each $i\in B$, it follows that $A_B\transp y_B =0, y_B > 0$ has a solution.
\end{proof}

We should note that,  depending on $A$, the set $N$ in Proposition~\ref{prop.partition} could be any subset of $\{1,\dots,m\}$.  In particular, $N = \emptyset$ if $A\transp y = 0$ for some $y > 0$, and $N = \{1,\dots,m\}$ if $Ax = 0$ for some $x>0$.  For instance, $N = \emptyset$ if $A = \matr{1\\-1}$ and $N = \{1,2\}$ if $A = 
\matr{1\\1}$.

\bigskip

Suppose $L\subseteq \R^n$ is a linear subspace and $K \subseteq \R^n$ is a closed convex cone.
Let
\begin{equation}\label{eq.hoffman}
\H(L,K) = \sup_{u \in \R^n\setminus L\cap K} \frac{\dist(u,L\cap K)}{\max\{\dist(u,L),\dist(u,K)\}},\end{equation}
with the convention that $\H(L,K) =0$ when $L\cap K = \R^n$.

In the remainder of this paper, we will use the following notation for $A\in \R^{m\times n}$: Let $B,N$ denote the canonical partition defined by $A$ as in Proposition~\ref{prop.partition} and let $L\subseteq \R^n, K\subseteq \R^n$ be defined as 
\[
L:=\{x: A_B x = 0\}, K := \{x: A_N x \le 0\},
\]
with the convention that $L = \R^n$ if $B=\emptyset$ and $K = \R^n$ if $N=\emptyset$.

Observe that $L$ is a linear subspace, $K$ is a closed convex cone, and $\{x: Ax\le 0\} = L\cap K$.  We now have all the necessary ingredients to upper bound $H_0(A)$.

\begin{theorem}\label{main.thm} Suppose $A\in \R^{m\times n}$ and 
the norm in $\R^m$ satisfies the componentwise compatibility condition. Let $B,N$ and $L,K$ be as above.  Then
\begin{equation}\label{eq.main.bound}
H_0(A) \le \H(L,K) \cdot \max\{H_0(A_N),H_0(A_B)\}.
\end{equation}
\end{theorem}
\begin{proof}
Suppose $u\in \R^n \setminus P$. The construction of $\H(\cdot,\cdot)$ and $H_0(\cdot),$ and the componentwise compatibility condition imply that there exists $x \in P=L\cap K$ such that 
\begin{align*}
\|x-u\| &\le \H(L,K)\cdot \max\{\dist(u,L),\dist(u,K)\} \\
&\le 
\H(L,K) \cdot \max\{H_0(A_B)\cdot \|(A_Bu)^+\|,H_0(A_N)\cdot \|(A_Nu)^+\|\} \\
&\le 
\H(L,K) \cdot \max\{H_0(A_B),H_0(A_N)\}\cdot \|(Au)^+\|. 
\end{align*}
Since this holds for all $u\in \R^n\setminus P$, the inequality in~\eqref{eq.main.bound} follows.  
\end{proof}
Observe that unlike $H_0(A)$ that depends on the {\em data representation} $A\in \R^{m\times n}$ of the cone $P=\{x: Ax\le 0\}$, the constant $\H(L,K)$ only depends on the sets $L\subseteq \R^n$ and $K\subseteq \R^n$.  In particular, $\H(L,K)$ does not depend on the norm in $\R^m$ while $H_0(A)$ evidently does. 

\medskip

The next proposition provides an upper bound on $\H(L,K)$ analogous to the upper bounds on $H_0(A)$ in Proposition~\ref{prop.case.1} and Proposition~\ref{prop.case.2}.  It will be useful for the computational procedure in Section~\ref{sec.computation}.

\begin{proposition}\label{prop.case.3}  Suppose $L\subseteq \R^n$ is a linear subspace and $K \subseteq \R^n$ is a closed convex cone. Then
\[
\H(L,K) \le 1 + 2\cdot \max_{u\in \R^n\atop \|u\|\le 1} \min_{x\in L, y\in K\atop x-y = u} \|x\|.
\]
\end{proposition}
\begin{proof} To ease notation, let $$H:=\max_{u\in \R^n\atop \|u\|\le 1} \min_{x\in L, y\in K\atop x-y = u} \|x\|.$$ We need to show that $\H(L,K) \le 1 + 2H.$
To that end, suppose $u \in \R^n\setminus L\cap K$.  Let $u_L:=\argmin_v\{\|u-v\|: v \in L\}$ and $u_K:=\argmin_v\{\|u-v\|: v \in K\}$.  The construction of $H$ implies that there exist $x\in L, y\in K$ such that $\|x\| \le H\cdot \|u_K-u_L\|$ and
$
x-y = u_K-u_L.
$
Hence
$
u_L+x = u_K+y \in L\cap K
$
and 
\begin{align*}
\dist(u,L\cap K) &\le \|u-u_L-x\|  \\
&\le \|u-u_L\| + \|x\| \\
&\le \|u-u_L\| + H\cdot \|u_K-u_L\| \\
& \le \max\{\dist(u,L),\dist(u,K)\} + H \cdot (\dist(u,K) + \dist(u,L))\\
&\le (1+2 H) \cdot\max\{\dist(u,L),\dist(u,K)\}.
\end{align*}
Since this holds for any $u\in \R^n\setminus L\cap K$, it follows that
\[
\H(L,K) \le 1+ 2H.
\]
\end{proof}

For computational purposes, it is useful to note that if $\bar x \in L \cap \text{\rm int}(K)$ is such that $\bar x + u \in K$ for all $\|u\|\le 1$ then Proposition~\ref{prop.case.3} implies that
\[
\H(L,K) \le 1 + 2\|\bar x\|.
\]

\section{A computable procedure to bound $H_0(A)$}
\label{sec.computation}
We next describe a procedure to compute an upper bound on $H_0(A)$.  The procedure consists of four main steps.  First, compute the partition $B,N$.  Second, compute an upper bound on $H_0(A_B)$.  Third, compute an upper bound on $H_0(A_N)$.  Fourth, compute an upper bound on $\H(L,K)$.  An upper bound on $H_0(A)$ thereby follows from Theorem~\ref{main.thm}.  For computational convenience, throughout this section we assume that $\R^m$ is endowed with the $\ell_\infty$ norm and $\R^n$ is endowed with the $\ell_2$ norm.  A Python implementation and some illustrative examples of this procedure are publicly available at
\begin{center}
{\tt https://github.com/javi-pena} 
\end{center}

\subsection*{Step 1: partition $B,N$}

The partition $B,N$ can be obtained from any point $(x,y,s,t)$ that satisfies the following systems of equations and inequalities for some $t > 0$:
\begin{equation}\label{eq.BN}
\begin{array}{rl} 
&A\transp y = 0\\
& Ax +s = 0 \\
& y + s -t\1 \ge 0 \\
&\1\transp y + \1\transp s = 1\\
&y \ge 0, s \ge 0.
\end{array}
\end{equation}
More precisely, if $(x,y,s,t)$ satisfies~\eqref{eq.BN} with $t>0$ then $B,N$ can be obtained as follows:
\[
B:=\{i: y_i >0\}, \; \; N:=\{i:s_i > 0\}.
\]
Proposition~\ref{prop.partition} guarantees that a solution $(x,y,s,t)$ to~\eqref{eq.BN} with $t>0$ always exists and that the associated partition $B,N$ is unique.  Such a point $(x,y,s,t)$ can be computed via the following linear program:
\begin{equation}\label{eq.LP}
\begin{array}{rl} \dmax_{x,y,s,t} & t \\
&A\transp y = 0\\
& Ax +s = 0 \\
& y + s -t\1 \ge 0 \\
&\1\transp y + \1\transp s = 1\\
&y \ge 0, s \ge 0.
\end{array}
\end{equation}

\subsection*{Step 2: upper bound on $H_0(A_N)$}
Suppose $N \ne \emptyset$ as otherwise $H_0(A_N) = 0$.
The remarks following Proposition~\ref{prop.case.1} show that \[H_0(A_N)\le \|\bar x\|_2\] for any $\bar x\in \R^n$ such that $A_N\bar x \ge \1$.   The best such upper bound can be computed via the following quadratic program
\begin{equation}\label{eq.QP.N}
\bar x := \argmin\{ \|x\|^2_2 : A_Nx \ge \1\}.
\end{equation}

\subsection*{Step 3: upper bound on $H_0(A_B)$}
Suppose $B \ne \emptyset$ as otherwise $H_0(A_B) = 0$.
The remarks following Proposition~\ref{prop.case.2} show that 
\[
H_0(A_B) \le \frac{2}{\sigma_{\min}^+(A_B\transp \bar Y)}
\]
for any $\bar y\in \R^B_{++}$ such that $\1_B\transp \bar y = 1$ and $A_B\transp \bar y = 0$.  Although the best such upper bound is challenging to compute, an upper bound of this kind that is within a factor of $\sqrt{|B|}$ of the best possible one can be computed via the following convex program
\begin{equation}\label{eq.CP.B}
\bar y := \argmin_{y\in \R^B_{++}}\left\{-\sum_{i\in B} \log(y_i): \1_B\transp y = 1, A_B\transp y = 0\right\}.
\end{equation}

\subsection*{Step 4: upper bound on $\H(L,K)$}
Suppose both $N \ne \emptyset$ and $B \ne \emptyset$ as otherwise $\H(L,K)=1$ or $\H(L,K) = 0$.
Let $Q$ be an orthonormal basis for $L:=\{x:A_Bx =0\}$ and $M = DA_NQ$ where $D$ is the diagonal matrix with positive diagonal entries such that all rows of $DA_N$ have Euclidean norm equal to one. Then the remarks following Proposition~\ref{prop.case.3} imply that
\[
\H(L,K) \le 1+2\|\bar z\|_2
\]
for any $\bar z \ge 0$ such that $M\bar z \ge \1$. The best such upper bound can be computed via the following  quadratic program
\begin{equation}\label{eq.QP.H}
\bar z := \argmin\{ \|z\|^2_2 : Mz \ge \1\}.
\end{equation}

\subsection*{Putting it all together: a procedure to bound $H_0(A)$}

Theorem~\ref{main.thm} allows us to stitch together the partition $B,N$ and the upper bounds on $H_0(A_B),$ $H_0(A_N),$ and $\H(L,K)$ to obtain an upper bound on $H_0(A)$ as detailed in Algorithm~\ref{the.algo} below.

\begin{algorithm}
\caption{Upper bound on $H_0(A)$}\label{the.algo}
\begin{algorithmic}[1]
\State {\bf input:}  $A \in \R^{m\times n}\setminus \{0\}$
\State  solve~\eqref{eq.LP} to enough accuracy to get a solution $(x,y,s,t)$ to~\eqref{eq.BN} with $t>0$ 
\State Let $B:=\{i: y_i > 0\}, N:=\{i:s_i > 0\}$
\If{$N\ne \emptyset$}
\State solve~\eqref{eq.QP.N} to enough accuracy to get $\bar x \in \R^n$ such that  $A_N\bar x \ge \1$
\EndIf
\If {$B\ne \emptyset$}
\State solve~\eqref{eq.CP.B} to enough accuracy to get $\bar y \in \R^B_{++}$ such that  $\1_B\transp \bar y =1$ and $A_B\transp \bar y = 0$
\EndIf
\If{$N=\emptyset$} return the upper bound 
$$
H_0(A) \le \frac{2}{\sigma_{\min}^+(A_B\transp \bar Y)}
$$ 
\EndIf
\If{$B=\emptyset$}
 return the upper bound 
$$
H_0(A) \le \|\bar x\|_2
$$ 
\EndIf

\State let $Q$ be an orthonormal basis for $L:=\{x:A_Bx =0\}$ and $M = DA_NQ$ where $D$ is the diagonal matrix with positive diagonal entries such that all rows of $DA_N$ have Euclidean norm equal to one
\State solve~\eqref{eq.QP.H} to enough accuracy to get $\bar z$ such that $ M\bar z \ge \1$
\State return the upper bound 
\[
H_0(A) \le (1+2\|\bar z\|_2) \cdot \max\left\{\|\bar x\|_2,\frac{2}{\sigma_{\min}^+(A_B\transp \bar Y)} \right\}
\]
\end{algorithmic}
\end{algorithm}

\section{Funding, conflict of interest, and data information}

This research has been supported by the Bajaj Family Chair at the Tepper School of Business.  The author has no competing interests to declare that are relevant to the content of this article.  The developments on this paper followed a theoretical mathematical approach and did not analyze or generate any datasets.

\bibliographystyle{plain}

\end{document}